\documentclass[11pt,a4paper]{amsart} 
\usepackage[all]{xy}

\usepackage{braket}
\SelectTips{cm}{}
\calclayout \makeatletter
\def\serieslogo@{} 
\def\@setcopyright{} 
\makeatother

\title[An upper bound for the finitistic dimension of an EI category algebra]{An upper bound for the finitistic dimension of an EI category algebra}

\author{Karsten Dietrich}
\address{Karsten Dietrich\\ Institut f\"ur Mathematik\\
Universit\"at Paderborn\\ D-33095 Paderborn\\ Germany.}
\email{karsten.dietrich@math.upb.de}


\newtheorem{lem}{Lemma}[section]
\newtheorem{prop}[lem]{Proposition}

\newtheorem{thm}[lem]{Theorem}

\newtheorem*{MThm}{Main~Theorem}

\newtheorem{Cor}[lem]{Corollary}

\theoremstyle{definition}
\newtheorem{exm}[lem]{Example}
\newtheorem{Rem}[lem]{Remark}
\newtheorem{defn}[lem]{Definition}

\numberwithin{equation}{section}

\renewcommand{\mod}{\operatorname{mod}\nolimits}
\newcommand{\Aut}{\operatorname{Aut}\nolimits}
\newcommand{\findim}{\operatorname{fin.dim}\nolimits}
\newcommand{\Ob}{\operatorname{Ob}\nolimits}

\newcommand{\pd}{\operatorname{proj.dim}\nolimits}

\newcommand{\Mod}{\operatorname{Mod}\nolimits}
\newcommand{\Findim}{\operatorname{Fin.dim}\nolimits}
\renewcommand{\mod}{\operatorname{mod}\nolimits}

\newcommand{\gldim}{\operatorname{gl.dim}\nolimits}

\def\C{{\mathcal C}}
\def\D{{\mathcal D}}

\def\P{{\mathcal P}}

\begin{document}

\begin{abstract}
EI categories can be thought of as amalgams of finite posets and finite groups and therefore the associated algebras are built up from incidence algebras and group algebras of finite groups. For this particular class of algebras we present a construction of an upper bound for the finitistic dimension which is due to L\"uck.
\end{abstract}

\maketitle

\section{Introduction}
The finitistic dimensions of a ring $\Lambda$ provide a measure for the complexity of the module category of $\Lambda$. They are defined as
\begin{eqnarray*}
	\findim (\Lambda) &=& \sup \Set{ \pd M | M \in \mod (\Lambda), \pd M < \infty } \\
    \Findim (\Lambda) &=& \sup \Set{ \pd M | M \in \Mod (\Lambda), \pd M < \infty }. 
	\end{eqnarray*}
	There are at least two canonical questions that arise in studying these invariants, namely: Are these two dimensions finite for any ring $\Lambda$ and do they coincide? For noetherian rings both questions have to be answered in the negative, but for finite dimensional algebras there is no counterexample up to now. In 1960 Bass published the two questions for finite dimensional algebras as ``problems''   and they are nowadays known as the finitistic dimension conjectures. 
	 
The little finitistic dimension, $\findim$,  is known to be finite for certain classes of algebras, for example for algebras with representation dimension at most $3$, monomial algebras or algebras with radical cube zero. One may consult \cite{Birge} for a survey on this conjecture and other homological conjectures (not including the result of Igusa and Todorov from \cite{IT} concerning the relation of the representation dimension and the finitistic dimension). 

We will provide an upper bound for the finitistic dimensions of EI category algebras which form an interesting class of finite dimensional algebras. Their representation theory has recently been studied by Webb \cite{Webb} and Xu \cite{Fei1}. Originally these algebras arose in algebraic topology in the 1970s and certain properties of their module categories have been studied by L\"uck and tom Dieck \cite{Lueck}, for instance, they described the indecomposable projective and the simple modules for EI category algebras. Their finitistic dimension (or the projective dimension of some modules over EI category algebras) plays a prominent role for the computation of certain invariants in the theory of finite $G$-spaces in algebraic topology; see \cite{Grodal}.
Throughout fix a field $k$ of characteristic $p \geq 0 $ (not necessarily algebraically closed). The main result is the following.
\begin{MThm}[L\"uck, {\cite[Proposition 17.31]{Lueck}} ]
Let $\C$ be an EI category and $k\C$ its associated unital $k$-algebra. Then the global dimension of $k\C$ is finite if and only if $|\Aut(x)|$ is invertible in $k$ for any $x \in \Ob\C$ and \[\Findim(k\C) \leq \ell(\C), \] where $\ell(\C)$ is the maximal length of a chain of non-isomorphisms in $\C$. 
\end{MThm}
It seems that this result is not well-known even among specialists. I would like to thank Jesper Grodal for his comments on a previous version of this manuscript, where I claimed that the above result is new.

\section{Basic facts about EI categories and their representations}
\begin{defn}
A category $\C$ in which every endomorphism is an isomorphism is called \emph{EI category}. We will only consider finite EI categories, i.e. those with only finitely many morphisms, to which we associate a finite dimensional, unital $k$-algebra $k\C$ having as $k$-basis the set of morphisms in $\C$ with the natural summation and multiplication given by the composition of morphisms if possible and zero otherwise.
\end{defn}
In order to avoid technicalities we will always assume that our EI categories are connected. All the results and proofs also work for the general case, but one has to restrict to the connected components at certain stages.
\begin{exm}
\begin{itemize}
\item[(1)] Let $G$ be a finite group and let $\underline{G}$ be the category with one object $x$ and $\text{End}(x)= G$. Then $\underline{G}$ is an EI category and $k\underline{G} = kG$.
\item[(2)] Let $Q$ be a finite quiver without oriented cycles and $\underline{Q}$ its path category. Then $\underline{Q}$ is an EI category (since it only has the identities as endomorphisms) with $k\underline{Q} = kQ$.
\item[(3)] If $(X,\leq)$ is a finite partially ordered set, then the associated category $\underline{X}$ is an EI category (for the same reason as above) and its algebra is the incidence algebra of the poset $(X,\leq)$.
\end{itemize}
\end{exm}
\begin{Rem}
If $\C$ is an EI category, one has a natural preorder defined on the set of objects $\Ob\C$, given by
\[ x \leq y \Leftrightarrow \C(x,y) \neq 0. \]
This preorder clearly induces a partial order on the set of isomorphism classes of objects of $\C$.
\end{Rem}
By $\Mod k\C$ we mean the category of left $k\C$-modules which can be identified with the category $\text{Fun}(\C,\Mod k)$ of all covariant functors from $\C$ into the category of vector spaces over $k$. The restriction of this identification yields an equivalence of categories $\mod k\C \simeq \text{Fun}(\C,\mod k\C)$ and is analogous to the equivalence of the module category of the path algebra of a given quiver and the category of representations of this quiver or its analogue for finite groups. We will switch between the concepts of modules on the one hand and functors on the other hand frequently in order to simplify the arguments.

There are classical theorems due to L\"uck that describe the structure of the simple and the indecomposable projective modules over an EI category algebra, which we want to mention here. One should be aware that also Auslander studied functor categories in his work on the representation theory of Artin algebras \cite{Auslander} and already knew the projective and simple objects of these categories, which in the special case of an EI category have been classified by L\"uck in a very concrete and convenient way.
\begin{prop}[L\"uck, \cite{Lueck}]
Let $\C$ be an EI category. Then any projective $k\C$-module is isomorphic to a direct sum of indecomposable projectives of the form $k\C \cdot e$  with $e \in k\text{Aut}(x)$ being a primitive idempotent for some $x \in \Ob \C$.
\end{prop}
For an object $x$ in $\C$ we denote by $[x]$ the isomorphism class of $x$. With this notation there is the following theorem.
\begin{thm}[L\"uck, {\cite[Theorem 9.39, Corollary 9.40]{Lueck}}]
Let $\C$ be an EI category. For each object $x \in \Ob\C$ and each simple $k\Aut(x)$-module $V$, there is a simple $k\C$-module $M$ such that $[x]$ is exactly the set of objects on which $M$ is non-zero and $M(x) = V$. Conversely, if $M$ is a simple $k\C$-module, then there is a unique isomorphism class of objects $[x]$ on which $M$ is non-zero and each $M(x)$ is a simple $k\Aut(x)$-module. This yields a bijection between the isomorphism classes of simple $k\C$-modules and the pairs $([x],V)$ where $x$ is an object in $\C$ and $V$ a simple $k\Aut(x)$-module.
\end{thm}
In view of this theorem, it is natural to denote a simple $k\C$-module by $S_{x,V}$ if it corresponds to the pair $([x],V)$ and to write $P_{x,V}$ for its projective cover. Note that the structure of $P_{x,V}$ is determined by its value at $x$.\\
Another important feature coming along with an EI category is the restriction functor to the category of representations of a subcategory of the EI category. Namely, let $\D $ be a subcategory of $\C$, then one has a functor $\downarrow^{\C}_{\D}:\mod k\C \to \mod k\D$  given by restriction of representations. This functor has a left adjoint which is usually denoted by $\uparrow^{\C}_{\D}= k\C \otimes_{k\D} -$. In \cite{Fei1}, Xu developed a theory of vertices and sources with respect to this restriction and induction which is analogous to the one for representations of finite groups. One should note that the restriction to a subgroup does not occur as a special case here, since a subgroup is not represented by a full subcategory in our setting.

\section{Construction of the upper bound}
Our aim is to give a complete and self-contained proof for the existence of an upper bound for the finitistic dimension of an EI category algebra $k\C$. Therefore we are going to recall some concepts and results concerning projective resolutions of modules over EI category algebras, most of them being due to  Xu. For the convenience of the reader we will also give proofs for these results (sometimes just special cases of the original ones with easier proofs). We would like to emphasize that  all this already occurs implicitly in \cite{Lueck}, but from a representation-theoretic point of view it is more convenient to use the notation of Xu. 
\begin{defn} Let $\C$ be an EI category.
\begin{itemize}
\item[(1)] Let $x$ be an object in $\C$. Then we define $\C_{\leq x}$ to be the full subcategory of $\C$ consisting of all objects $y \in \Ob\C$ with $\C(y,x) \neq \emptyset$. Similarly we define $\C_{\geq x}$.
\item[(2)] An ideal in $\C$ is a full subcategory $\D$ of $\C$ such that for any object $x$ in $\D$ we have that $\C_{\leq x} \subseteq \D$. 
\item[(3)] Let $M$ be a $k\C$-module. The $M$-minimal objects are the objects $x \in \Ob\C$ such that $M(x) \neq 0$ and for any $y \in \Ob\C$ with $y \not\simeq x$ and $\C(y,x)\neq \emptyset$ one has $M(y) = 0$.
\item[(4)] Let $M$ be again a $k\C$-module. We put $\C_{M}$ to be the full subcategory consisting of all $y \in \Ob\C$ with $\C(x,y) \neq \emptyset$ for some $M$-minimal object $x$ in $\C$.\\
\end{itemize}
\end{defn}
It is clear by definition that any $k\C$-module $M$ is determined by its values on $\C_M$. We are now going to point out a nice property of ideals, namely that in this case the restriction preserves projectives.
\begin{prop}[Xu, {\cite[Lemma 3.1.6]{Fei1}}]\label{res}
Let $\D$ be an ideal in $\C$. Then the restriction functor $\downarrow^{\C}_{\D}$ preserves projective (left-)modules.
\end{prop}
\begin{proof}
We define a functor $F:\Mod k\D \to \Mod k\C$ by $FM(x) = M(x)$ for $x \in \Ob \D$ and $FM(x) = 0$ for $x \not\in \Ob \D$. On morphisms $F$ is also just given by filling up with zeros. This clearly defines a functor which is by definition exact. Since $\D$ is supposed to be an ideal we get that this functor $F$ is a right adjoint to the restriction functor. Therefore the restriction has to preserve projective modules.
\end{proof}
With this result we obtain important properties for projective resolutions of $k\C$-modules, which are again due to Xu.

\begin{lem}[Xu, {\cite[Lemma 4.1.1]{Fei1}}]
Let $M$ be a $k\C$-module and $P_M$ its projective cover. Then $P_M$ is supported on $\C_M$ and for any $M$-minimal object $x$ the module $P(x)$ is the projective cover of $M(x)$.
\end{lem}
\begin{proof}
By the characterization of the indecomposable projective modules and by definition of $\C_{M}$ it is clear that the support of $P_{M}$ is contained in $\C_{M}$.
Now let $x$ be an $M$-minimal object in $\C$. The full subcategory $\C_{\leq x}$ is an ideal in $\C$ and its intersection with $\C_M$ is just $x$. Now by Proposition \ref{res} the $k\Aut(x)$-module $P_M (x) $ is a projective module and admits a surjection onto $M(x)$. Thus, what remains to be shown is the minimality of $P_M (x)$. If $P_M(x)$ would not be the projective cover of $M(x)$, then by the universal property of the projective cover, there would exist projective modules $P_1$ and $P_2$ such that $P_M = P_1 \oplus P_2$ with $P_1(x)$ being the projective cover of $M(x)$ and $P_2(x) \neq 0$ such that, if $\pi: P_M \to M$ is the defining essential epimorphism, then $\pi \downarrow^{\C}_{\D} $ sends $P_2(x)$ to zero. The module $P_{2}$ has an indecomposable projective direct summand $P_{2}'$.  Now, since $P_{2}'(x) = k\Aut(x) \cdot e$ for some primitive idempotent $e \in k\Aut(x)$ and $P_{2}' = k\C \cdot e$ it follows that $\pi$ sends $P_{2}'$ to zero. This is a contradiction to the minimality of $P_M$.
\end{proof}
This Lemma gives us the following description of the minimal projective resolution of a $k\C$-module $M$.
\begin{Cor}\label{resolution}
Let $M$ be a $k\C$-module and $\P_M$ a minimal projective resolution. Then $\P_M$ is supported on $\C_M$ and for any $M$-minimal $x \in \Ob\C$ we have that $\P_M(x)$ is a minimal projective resolution of $M(x)$.
\end{Cor}
With these preparations we are now in the position to prove the Main Theorem.
\begin{proof}[Proof of the Main Theorem]
We will only prove the assertion concerning the finitistic dimension. A proof for the assertion on the global dimension can be found in \cite{Fei1}.

Let $M$ be a $k\C$-module which is of finite projective dimension and consider a minimal projective resolution
\[\P_M: \quad 0 \to P_n \to P_{n-1} \to \dots \to P_1 \to P_0 \to M \to 0. \]
Then $\P_M$ is supported on $\C_M$ as we have seen above. Now let $x$ be an $M$-minimal object in $\C$. By Corollary~\ref{resolution} $\P_M(x)$ is a minimal projective resolution of $M(x)$ as $k\Aut(x)$-module. As a module over a group algebra of a finite group, the module $M(x)$ is either projective or of infinite projective dimension. The latter case is impossible, since $\P_M$ is a finite resolution. This implies, that $P_1(x) = 0$ and $P_1$ is supported on $\C_M \setminus \lbrace M\text{-minimal objects}\rbrace $. Applying this argument inductively to any of the $P_i$, we get that $n \leq \ell(\C)$ and hence the claim.
\end{proof}
Finally we will present some examples to illustrate the main result.
\begin{exm}
\begin{itemize}
\item[(1)] Let $\C$ be an EI category with only one object. Then $k\C$ is the group algebra of a finite group $G$. For this case it is well known that the finitistic dimension is zero and $\ell(\C)$ equals zero as well.
\item[(2)] Suppose that $\C$ is the path category of a finite quiver without oriented cycles, then $k\C$ is hereditary and therefore $\findim (k\C)  = \gldim (k\C) \leq 1$. Thus, in this case the given bound is not optimal.
\item[(3)] Let \[ \C: \begin{xy}\xymatrix{ X \ar@(ul,dl)_{g} \ar[r]^{f} & Y \ar@(ur,dr)^{1_Y} } \end{xy}\] 
be the EI category given by the relations $g^2 = 1_X$ and $fg = f$. Further suppose that $k$ is a field of characteristic $2$. Then the indecomposable projective representations of $\C$ are exactly
 \[ P_X: \begin{xy}\xymatrix{ k^2 \ar@(ul,dl)_{A} \ar[r]^{(1\; 1)} & k \ar@(ur,dr)^{(1)} } \end{xy} \mbox{ where } A = \left(\begin{smallmatrix} 0 & 1 \\ 1 & 0 \end{smallmatrix}\right)\] and
  \[ P_Y: \begin{xy}\xymatrix{ 0 \ar@(ul,dl)_{(0)} \ar[r]^{(0)} & k \ar@(ur,dr)^{(1)} } \end{xy}.\] 
Now it is clear that this algebra is of infinite global dimension since the group algebra $k\Aut(x)$ is not semisimple. Its finitistic dimension equals $1$ since it is easy to see, that the representation
 \[ M: \begin{xy}\xymatrix{ k^2 \ar@(ul,dl)_{A} \ar[r]^{(0)} & 0 \ar@(ur,dr)^{(0)} } \end{xy} \text{ with } A = \left(\begin{smallmatrix} 0 & 1 \\ 1 & 0 \end{smallmatrix}\right) \] 
 has projective dimension $1$ and there is no module with projective dimension greater than one (if the projective dimension is finite) by our theorem.
\end{itemize}
\end{exm}
 \bigskip

\addcontentsline{toc}{section}{References}
\nocite{*}
\bibliographystyle{plain}
\bibliography{Refs_note}

\end{document}